\documentclass{article}
\usepackage[a4paper, total={6in, 8in}]{geometry}
\usepackage[utf8]{inputenc}
\usepackage[english]{babel}
\usepackage[all,cmtip]{xy}
\usepackage{csquotes}
\usepackage{tikz-cd}
\usepackage[T1]{fontenc}    
\usepackage{hyperref}       
\usepackage{url}            
\usepackage{booktabs}       
\usepackage{amsthm}
\usepackage{amsmath}
\usepackage{amsfonts}       
\usepackage{nicefrac}       
\usepackage{microtype} 
\usepackage[T1]{fontenc}
\usepackage{lmodern}
\newtheorem{theorem}{Theorem}
\newtheorem{corollary}[theorem]{Corollary}
\newtheorem{lemma}{Lemma}

\newtheorem{prop}[theorem]{Proposition}
\theoremstyle{definition}
\newtheorem{definition}{Definition}
\newtheorem{eg}{Example}

\newcommand{\mC}{\mathcal{C}}

\newcommand{\mD}{\mathcal{D}}
\newcommand{\mE}{\mathcal{E}}
\newcommand{\mI}{\mathcal{I}}

\newcommand{\mN}{\mathfrak{N}}
\newcommand{\mR}{\mathcal{R}}

\newcommand{\mF}{\mathbb F}

\def \no {\noindent}
\def \R{\mathbb{R}}

\newcommand{\N}{\mathbb{N}}
\def \C {\mathcal{C}}
\newcommand{\E}{\mathcal{E}}
\def \U {\mathcal{U}}

\def \D{\mathcal{D}}
\def \supo{\operatorname{supp}}

\newcommand{\ring}[1]{\mathcal{R}_{#1}}

\newcommand{\cod}{\mathfrak{C}}

\newcommand{\NC}{\mathfrak{N}}
\newcommand{\NN}{\mathbb{N}}

\newcommand{\F}{\mathcal{F}}
\newcommand{\G}{\mathcal{G}}

\begin{document}
	
	\title{Neural category}
	\author{Neha Gupta and Suhith K N\thanks{Suhith's research is partially supported by Inspire fellowship from DST grant IF190980.} \\
		Department of Mathematics\\
		Shiv Nadar Insititution of Eminence\\
		Uttar Pradesh - 201314 \\
		\texttt{neha.gupta@snu.edu.in, sk806@snu.edu.in} \\ \\ }
\date{}
	\maketitle
	\begin{abstract}
		A neural code on $ n $ neurons is a collection of subsets of the set $ [n]=\{1,2,\dots,n\} $. Curto et al. \cite{curto2013neural} associated a ring $\ring{\mathcal{C}}$ (neural ring) to a neural code $\mathcal{C}$. A special class of ring homomorphisms between two neural rings, called neural ring homomorphisms were introduced by Curto and Youngs \cite{curto2020neural}.   The main work in this paper comprises of constructing two categories. First is the $\mathfrak{C}$ category, which is a sub category of SETS consisting of neural codes and code maps. Second is the neural category $\mathfrak{N}$, which is a subcategory of \textit{Rngs} consisting of neural rings and neural ring homomorphisms. Then, the rest of the paper characterizes the properties of these two categories like initial and final objects, products, coproducts, limits, etc. Also, we show that these two categories are in dual equivalence.
		
	\end{abstract}
	
	\section{Introduction}
	A neural code on $n$ neurons, denoted by $ \C $, is a collection of subsets of the  set $ [n]=\{1,2,3,\dots,n\}. $  One can also see neural code as binary strings and throughtout this paper we will consider neural codes as binary strings defined below: 

\begin{definition}\cite{curto2020neural}
	A $neural$ $code$ on $n$ neurons is a set of binary strings of length $n$ for some  $ n\in \N $.  The elements(binary strings) of a neural code are called its $code$ $words$. So, given a neural code $\C$ on $n$ neurons we can think of it as $\C\subseteq \{0,1\}$.
\end{definition}

 Given neural codes $\mathcal{C} \subseteq \{0, 1\}^n$ and $\mathcal{D} \subseteq \{0, 1\}^m$, on $n$ and $m$ neurons respectively, a $code$ $map$ is any function
$q : \mC \to \mD$ sending each code word $c\in \mC$ to another codeword $q(c) \in\mD$.

The importance of neural codes comes from the discovery of place cells in the hippocampus of rats by O’Keefe and Dostrovsky in 1971. Their discovery highlighted that cells in the rat's hippocampus fire in specific locations of the rat's environment.  Place cells are neurons that are essential for the rat's ability to perceive space. Since at every particular location there are just few cells that fire. Thus the binary string we obtain when we consider certain $n$ place cells at a particular environment becomes a codeword. Further, if we consider binary strings for the complete environment we obtain a neural code. Consider $n$ neurons and $\U=\{U_1,\dots, U_n\}$ be a a collection of sets in $\R^k$, i.e., for all $i \in [n]$ we have $U_i\subseteq \R^k$. Let $U_i$ be the location in which $i\text{th}$ neuron fires. Then the neural code obtained from this environment is given by $$\C(\U)= \left\{c=c_1^{}\dots c_n^{}\in\{0,1\}^{n} \ \ \Big\vert\ \displaystyle\bigcap_{j\in \supo(c)} U_j \setminus \displaystyle\bigcup_{i\not\in\supo(c) } U_i \not= \phi \right\},$$ where $\supo(c)=\{i\in[n]\mid c_i=1\}.$ Experimental results \cite{curto2017makes} show that the $U_i$'s, i.e., a specific environment where the neuron fired was approximately open convex sets in $\R^2$. 

Next, the fundamental research question in this area is that given a collection of binary strings of length $n$, or in other words given a neural code on $n$ neurons does there exists a collection of sets $\U=\{U_1,\dots,U_n\}$ in a euclidean space $\R^k$ for some $k>0$, such that $\C=\C(\U)$. Further, the experimental data (mentioned above) motivates to ask whether there exist collection of open convex sets, $\U$ that satisfies the condition $\C=\C(\U).$ However, this is not true, for example the code $\C=\{110,011,100,001,000\}$ does not have a open convex realization (Refer Example 2.1 \cite{jeffs2020morphisms}). However, it is also known that every neural code is at least convex realizable. This result was proved by Franke and Muthiah \cite{franke2018every} in 2019.

Further, the goal shifted in classifying the neural codes into open convex and not open convex. This motivated in introducing the algebraic direction to neural codes. Curto et al. \cite{curto2013neural} introduced neural ring for every neural code.  For any neural code $\mathcal{C} \subseteq \{0, 1\}^n$,they defined the associated ideal $\mI_\mC \subseteq \mF_2[x_1, \dots, x_n]$ as follows:
$$ \mI_\mC = \{f \in \mF_2[x_1, \dots, x_n]\mid f(c) = 0 \text{ for all } c \in \mC\}.$$
The $neural$ $ring$ $\mR_\mC$ is then defined as $R_\mC = \mF_2[x_1, \dots, x_n]/\mI_\mC$.
Note that considering a neural code $\mC$ to be the emptyset, the corresponding neural ring $\mR_\emptyset$ will simply be the singleton trivial ring. Further note that, the neural ring $\mR_\mC$ is precisely the ring of functions $\mC \to \{0, 1\}$, denoted $\mF_2^\mC$. Also, one can then obtain an immediate relationship between code maps and ring homomorphisms using the pullback map. Given a code map $q : \mC \to \mD$, each $f \in \mR_\mD$ is a function $f : \mD \to \{0, 1\}$,
and therefore one can pull back $f$ by $q$ to a function $f \circ q : \mC \to \{0, 1\}$, which is an element of $\mR_\mC$. Thus for any $q : \mC \to \mD$, it has a pullback map $q^{\star} : \mR_\mD \to \mR_\mC$, where $q^{\star}(f) = f \circ q$, as shown 
	\[
	\begin{tikzcd}
			& \{0,1\}\\
			\mC \arrow{ur}{q^*f = f\circ q}\arrow[swap]{r}{q} & \mD\arrow[swap]{u}{f}
		\end{tikzcd}
	\]
In fact, the pullback provides a bijection between code maps and ring homomorphisms, as mentioned in a proposition in \cite{curto2020neural}. 
\begin{prop}\cite[Proposition 2.2]{curto2020neural}
	There is a 1-1 correspondence between code maps $q : \mC \to \mD$ and ring homomorphisms $\phi : \mR_\mD \to \mR_\mC$, given by the pullback map. That is, given a code map $q : \mC \to \mD$, its pullback $q^{\star} : \mR_\mD \to \mR_\mC$ is a ring homomorphism; conversely, given a ring homomorphism $\phi : \mR_\mD \to \mR_\mC$,
	there is a unique code map $q_\phi : \mC \to \mD$ such that $(q_\phi)^{\star} = \phi$.
\end{prop}
So far, we have neural codes and code maps between them; associated neural ring, with  usual ring homomorphisms between them, and a correspondence between these objects. The above result surely guarantees the correspondence, but also reveals that the corresponding ring homomorphisms need not preserve any structure of the associated codes. This is so because any code map (simply any function between sets) will have a corresponding ring homomorphism. The next proposition tells us that even ring isomorphisms are not good enough to reveal any information about the corresponding code maps. 
\begin{prop}\cite[Proposition 2.3]{curto2020neural} A ring homomorphism $\phi : \mR_\mD \to \mR_\mC$ is an isomorphism if and only if the corresponding code map $q_\phi : \mC \to \mD$ is a bijection.
\end{prop}

The above two results express the connection between the neural codes on one hand and the neural rings on the other. However, they also highlight that the usual ring homomorphisms and even isomorphisms are not enough to get back any information about the code maps. From the last result one can conclude that the neural rings are rings of functions from $\mC$ to $\{0, 1\}$, and that its abstract structure
depends only on the number of codewords, $|\mC|$. Thus if one considers such rings abstractly, they express no additional structure of the corresponding neural code, not even the length of the codewords (or number of neurons, $n$). So, the question is what algebraic constraints can be put on homomorphisms between neural rings in order to trap some meaningful class of code maps?

Consider, for example, the code maps such as permutation and adding or removing trivial neurons. In these cases, the code maps act by preserving the activity of each neuron in a trivial way. 
Keeping these code maps in mind, we restrict to a class of maps in the category of rings, that respect the elements of the neural ring corresponding to individual neurons. We denote the individual neuron as a variable $x_i$. 

\begin{definition}\cite[Definition 3.1]{curto2020neural}
	Let $\mC \subseteq \{0, 1\}^n$ and $\mD \subseteq \{0, 1\}^m$ be neural codes, and let $\mR_\mC = \mF_2[y_1, \cdots, y_n]/\mI_\mC$
	and $\mR_\mD = \mF_2[x_1, \cdots, x_m]/\mI_\mD$ be the corresponding neural rings. A ring homomorphism $\phi : \mR_\mD \to \mR_\mC$
	is a $neural$ ring homomorphism if $\phi(x_j) \in \{y_i|i \in [n]\}\cup \{0, 1\}$ for all $j \in[m]$. We say that a neural ring homomorphism $\phi$ is a $neural$ $ring$ $isomorphism$ if it is a ring isomorphism and its inverse is 	also a neural ring homomorphism.
\end{definition}
 It is straightforward to see that the composition of neural ring homomorphisms is again a neural
ring homomorphism. This is a result given by Curto and Youngs \cite{curto2020neural},

\begin{lemma}\cite[Lemma 3.3]{curto2020neural} If $\phi: \mR_\mD \to \mR_\mC$ and $\psi : \mR_\mE \to \mR_\mD$ are neural ring homomorphisms, then their
	composition $\phi\circ\psi$ is also a neural ring homomorphism. Moreover, if $\phi$ and $\psi$ are both neural ring isomorphisms,
	then their composition $\phi\circ\psi$ is also a neural ring isomorphism. \label{lemma1}
\end{lemma}

Note that for a given neural code $\mC$, the identity ring homomorphism $\mR_\mC\to\mR_\mC$ is trivially a neural ring homomorphism. In fact, it is a neural ring isomorphism. Let us denote it as $I_\mC$ for a neural code $\mC$. This identity map sends $x_i$ to $y_i$ considering $I_\mC:\mR_\mC= \mF_2[x_1, \dots, x_n]/\mI_\C\to \mR_\mC= \mF_2[y_1, \dots, y_n]/\mI_\C$ with $|\mC|=n$ and $i\in[n]$.

The following theorem from \cite{curto2020neural} introduces the code maps, which yield neural ring homomorphisms. All of these code maps are useful in a neural context, and preserve the behavior of individual neurons. In fact, any neural ring homomorphisms correspond to code maps that are compositions of the following five elementary types of code maps. For proof, refer  \cite{curto2020neural}.

\begin{theorem}\cite[Theorem 3.4]{curto2020neural} A map $\phi : \mR_\mD \to\mR_\mC$ is a neural ring homomorphism if and only if $q_\phi$ is a composition of the following elementary code maps: \label{thmnrh}
	\begin{enumerate}
		\item Permutation
		\item Adding a trivial neuron (or deleting a trivial neuron)
		\item Duplication of a neuron (or deleting a neuron that is a duplicate of another)
		\item Neuron projection (deleting a not necessarily trivial neuron)
		\item Inclusion (of one code into another)
	\end{enumerate}
	Moreover, $\phi$ is a neural ring isomorphism if and only if $q_\phi$ is a composition of maps (1)-(3).
\end{theorem} 

It is worth mentioning the idea behind the proof. The authors define a key vector $V$ associated to any neural ring homomorphism which completely and uniquely determines it. Using this key vector one can get the corresponding code map. Conversely, if a code map is given which is described using a key vector $V$, then the associated ring homomorphism is neural with the key vector $V$. So, the entire point is to be able to get this key vector for a code map, which will ensure that the code map corresponds to a neural ring homomorphism and conversely. Let us now see the exact definition of the of key vectors.
\begin{definition}\cite[Definition 3.6]{curto2020neural}
	Let $\phi: \mR_\mD \to \mR_\mC$ be a neural ring homomorphism, where $\mC$ and $\mD$ are codes on $n$ and $m$ neurons, respectively. The key vector of $\phi$ is the vector $V\in\{1,\cdots,n,0,u\}^m$ such that
	$$V_j=
	\begin{cases}
		i & \text{ if }\phi(x_j)=y_i\\
		0 & \text{ if }\phi(x_j)=0\\
		u & \text{ if }\phi(x_j)=1
	\end{cases}.
	$$
\end{definition}
Here, we have denoted the multiplicative identity of the ring with the symbol $u$. 
\begin{eg}
	Let $\mC=\{00,10,01\}$ and $\mD=\{00,10\}$. Define a ring homomorphism $\phi:\mR_\mD\to\mR_\mC$ as 
	$$\phi(\rho_{00}^{})=\rho_{00}^{}$$
	$$\phi(\rho_{10}^{})=\rho_{10}^{}+\rho_{01}^{}.$$
	In $\mR_\mD$, $x_1=\rho_{10}^{}$, and in $\mR_\mC$ $y_1=\rho_{10}^{}$ and $y_2=\rho_{01}^{}$. So,
	$$\phi(x_1)=\rho_{10}^{}+\rho_{01}^{}\notin\{y_1^{},y_2^{},0,1\}.$$
	Thus, $\phi$ is not neural. Also, observe that one can not define the key vector for such a ring homomorphism. 
\end{eg}
Next we discuss couple of lemmas that link key vectors and code maps. 
\begin{lemma}\cite[Lemma 3.7]{curto2020neural}
	Let $\phi:\ring{\D}\to \ring{\C}$ be a neural ring homomorphism with key vector $V$. Then the corresponding code map $q_\phi:\C\to\D$ is given by $q_\phi(c)=d, $ where $d_j=\begin{cases}
		c_i \quad & \text{ if } V_j =i \\0 \quad & \text{ if } V_j =0 \\ 1 \quad & \text{ if } V_j =u.
	\end{cases}$
\end{lemma}
\begin{lemma}\cite[Lemma 3.8]{curto2020neural}
	Let $\C$ and $\D$ be codes on $n$ and $m$ neurons, respectively. Suppose $q:\C\to \D$ is a code map and $V\in\{1,\dots,n,0,u\}^m$ such that $q$ is described by $V$; that is for all $c\in\C,$ $q=d, $ where $d_j=\begin{cases}
		c_i \quad & \text{ if } V_j =i \\0 \quad & \text{ if } V_j =0 \\ 1 \quad & \text{ if } V_j =u.
	\end{cases}$ Then the associated ring homomorphism $\phi_q$ is a neural ring  homomorphism with key vector $V$. 
\end{lemma}
Next, we present an example of a code map and its corresponding neural ring homomorphism with key vector using the above lemma. 
\begin{eg}
	Let $\mC=\{001,110,101\}$ and $\mD=\{01,10,01\}$. Define a code map $q:\mC\to\mD$ which projects onto the 2nd and 3rd neurons (deleting first neuron). Then,  the key vector for this code map is $V=(2,3)$. And its corresponding neural ring homomorphism $\phi_q$ will have the same key vector. Thus using the key vector $V=(2,3)$, the corresponding $\phi_q:\mR_\mD\to\mR_\mC$ must be the one which maps
	$$\phi_q(x_1)=y_2\quad\text{and}\quad\phi_q(x_2)=y_3.$$
	Solving this gives the unique neural ring homomorphism $\phi_q$ : $\rho_{10}\mapsto\rho_{110}^{}$ and $\rho_{00}\mapsto\rho_{001}^{}+\rho_{101}^{}$.
\end{eg}
Finally, we give an example of a neural ring homomorphism with key vector, and using which we reconstruct the corresponding code map.
\begin{eg}
	Let $\mC=\{0000, 0001, 0010, 0011, 0100,0101,0110\}$ and  $\mD=\{0000,0001, 1001, 0010, 1010, \linebreak 0011, 1000\}$. Define a ring homomorphism $\phi:\mR_\mD\to\mR_\mC$ as follows:
	\begin{align*}
		\phi(\rho_{0000}^{}) &= \rho_{0000}^{} & \phi(\rho_{0001}^{}) &= \rho_{0010}^{}\\
		\phi(\rho_{1001}^{}) &= \rho_{0011}^{} & \phi(\rho_{0010}^{}) &= \rho_{0100}^{}\\
		\phi(\rho_{1010}^{}) &= \rho_{0101}^{} & \phi(\rho_{0011}^{}) &= \rho_{0110}^{}\\
		\phi(\rho_{1000}^{}) &= \rho_{0001}^{} 
	\end{align*}
	One can then easily calculate that 
	$$\phi(x_1^{})=y_4^{}\quad\quad \phi(x_2^{})=0$$
	$$\phi(x_3^{})=y_2^{}\quad\quad \phi(x_4^{})=y_3^{}$$
	Thus $\phi$ is neural. Note that key vector for $\phi$ is $V=(4,0,2,3)$. Using this key vector the corresponding $q:\mC\to\mD$ must be the one which maps \\ 
	
	\centering
	\begin{tabular}{cccccc}
		
		$ q: c_1^{}c_2^{}c_3^{}c_4$ & $ \mapsto $ & $ c_4 $ & 0 &$ c_2 $& $ c_3 $   \\
		\hline
		$ q:0000 $& $ \mapsto $ & 0 & 0 & 0 & 0 \\
		
		$ q:0001 $& $ \mapsto $ & 1 & 0 & 0 & 0 \\
		
		$ q:0010 $& $ \mapsto $ & 0 & 0 & 1 & 0  \\
		
		$ q:0011$& $ \mapsto $ & 1 & 0 & 0 & 1 \\
		
		$ q:0100 $& $ \mapsto $ & 0 & 0 & 1 & 0 \\
		
		$ q:0101 $& $ \mapsto $ & 1 & 0 & 1 & 0 \\
		$ q:0110 $& $ \mapsto $ & 0 & 0 & 1 & 1 \\
	\end{tabular} 
	
	
\end{eg}
\noindent So, the corresponding ring homomorphism $\phi_q^{}$ for this code map is the one that is given.
	
	\section{The neural category}

Let us first define a subcategory of SETS. Consider neural codes as the objects of this category, and morphisms as the finite compositions of elementary code maps (defined as in Theorem \ref{thmnrh}). With the usual function composition and the usual identity function it becomes a subcategory of SETS. We call it the code category, denoted by $\cod$. Clearly it is not full, for example, we may not be able to write a constant function as a composition of the elementary functions. 

With the constructions defined so far, we actually have enough data to put them together to define a category with rings. We consider neural rings as the objects and morphisms to be the neural ring homomorphisms. We call this category the $Neural$ category and denote it by $\mN$. Clearly, $\mN$ is a subcategory of the category $Rngs$ of rings and ring homomorphisms. 

\begin{prop}
	The collection $\mN$ of neural codes and neural ring homomorphisms  forms a category. In fact, it is a subcategory of $Rngs$, and 
	there is a faithful functor $\mN \to Rngs$.
\end{prop}
\begin{proof}
	The composition of morphisms is the usual composition of ring homomorphisms and it clearly preserves the extra conditions using Lemma \ref{lemma1}. The identity ring homomorphism is neural. Trivially $\mN$ is a subcategory of $Rngs$. Finally, the inclusion functor $\mN \to Rngs$ is faithful.
\end{proof}

Note that the inclusion functor mentioned is certainly not full. This is so because we have examples of ring homomorphisms between neural rings which may not be neural. Before we give an explicit example of such a ring homomorphism, let us go back to the definition of a neural ring and see what are the ways to denote a generic element. 

Given a neural code $\mC$, an element of the neural ring $\mR_\mC=\mF_2[x_1,\cdots,x_n]/\mI_\mC$ may be denoted in different ways. Firstly, it can be written as a representative of an equivalence class which would be some polynomial over $\mF_2$ mod $\mI_\mC$. Alternatively, as we said before, the neural ring $\mR_\mC$ is the ring of functions from $\mC$ to $\{0,1\}$ and thus has a vector space structure over $\mF_2$. Consequently, an element of $\mR_\mC$ can be written as a polynomial function $\mC \to \{0, 1\}$ defined completely by the codewords that support it. We can make use of this idea to form a canonical basis for $\mR_\mC$, that consists of characteristic functions $\{\rho_c \hspace{1mm}| \hspace{1mm}c \in\mC\}$ , where 
$$
\rho_c(v) = 
\left\{
\begin{array}{ll}
	1 & \text{if } v = c\\
	0 & \text{otherwise}
\end{array} 
\right.$$
In polynomial notation,
$$
\rho_c = 
{\displaystyle \prod_{c_i=1} x_{i} \prod_{c_j=0}(1 - x_j )}$$
where $c_i$ represents the $i$th component of codeword $c$. The characteristic functions $\rho_{c}$ form a basis
for $\mR_\mC$ as an $\mF_2$-vector space, and they have several useful properties. For further details and more properties, refer \cite{curto2013neural,curto2020neural}. We give a very generic example which exhibits the properties. Suppose we have just two variables, that is, we are talking of a code on 2-neurons. Now, $\mF_2[x_1^{},x_2^{}]=\left\{ \displaystyle\sum_{i,j=0}^n r_{ij}^{} x_1^i x_2^j\hspace{2mm}\big|\hspace{2mm} r_{ij}^{}\in\mF_2^{},n\in\NN \right\}.$ Consider $\mC=\{00,10,01\}$. Then, $\mI_\mC=\left\{ \displaystyle f=\sum r_{ij}^{} x_1^i x_2^j\mid f(c)=0 \hspace{1mm} \text{ for all } c\in\mC \right\}.$ Further, $\mR_\mC=\mF_2[x_1^{},x_2^{}]\big/\mI_\mC$, is a vector space over $\mF_2$ with a basis $\{\rho_{00}^{}, \rho_{10}^{}, \rho_{01}^{}\}$. Here $\rho_{00}^{}=(1-x_1^{})(1-x_2^{})$, $\rho_{10}^{}=x_1^{}(1-x_2^{})$ and $\rho_{01}^{}=(1-x_1^{})x_2^{}$. Needless to say that these polynomials are a representative of their equivalence class mod $\mI_\mC$. Each element $f$ of $\mR_\mC$ can be represented as the formal sum of these basis elements for the codewords in its support: ${\displaystyle f =\sum_{\{c\in\mC|f(c)=1\}}\rho_c^{}}$. In particular, $x_i^{} = {\displaystyle \sum_{\{c\in\mC|c_i^{}=1\}}\rho_c^{}}$. 
The identity of the neural ring is given by ${\displaystyle\sum_{c\in\mC}\rho_c^{}} =\rho_{00}^{}+ \rho_{10}^{} + \rho_{01}^{}= 1-x_1^{}x_2^{}$.

Coming back to giving an example for a ring homomorphism which is not neural, consider the following ring homomorphism defined on the basis elements.

\begin{eg}
	Let $\mD = \{00, 10\}$ and $\mC = \{010, 110\}$. Define the ring homomorphism $\phi : \mR_\mD \to \mR_\mC$ on basis elements as
	follows: $\phi(\rho_{00}) = \rho_{110},$ $\phi(\rho_{10}) = \rho_{010} $. In $\mR_\mD$, $x_1 = \rho_{10}$ and $x_2 = 0$. Thus, $\phi(x_1) = \phi(\rho_{10})= \rho_{010}$, which is not equal to any of the $y_i$'s. This is so because, $y_1 = \rho_{110}$, $y_2 = \rho_{010} + \rho_{110}$ and $y_3 =0$. Thus, $\phi$ is not a
	$neural$ ring homomorphism.
\end{eg}

Next we peep into the neural category  $\mN$ and explore if it has some interesting objects or properties. We first note that neural rings $\mR_\mC$ corresponding to any neural code $\mC$ of cardinality one, are isomorphic objects in $\mN$. Let $ \C $ and $ \D $ be two codes with $ \vert\C\vert =\vert \D\vert =1$. Then we observe that there exists $ q: \C\to\D $ such that $ q $ is a composition of maps in (1)-(3) of Theorem  \ref{thmnrh}. Therefore, the associated ring homomorphism $ \phi_q:\ring{\D}\to\ring{\C} $ is a neural ring isomorphism. Hence  the neural rings corresponding to codes of cardinality one are all isomorphic in the category $ \NC. $ Let us represent this class with  $\mR_{\{1\}}$. Then $\mR_{\{1\}}$ is the initial object in $\mN$ (of course, up to isomorphism). Note that a singleton set in category of sets is a terminal object. Moreover, empty set is the unique initial object in category of sets, and we show that $\mR_{\emptyset}$ is the unique final object in $\mN$.

\begin{prop}
	Neural ring $\mR_{\{1\}}$ corresponding to any neural code of cardinality one is the initial object in $\mN$. Moreover, $\mR_{\emptyset}$ is the unique final object in $\mN$. \label{propintial}
\end{prop}
\begin{proof}
	Consider a neural code $\mC$, and any neural ring homomorphism $\phi:\mR_{\{1\}}\to\mR_\mC$. Let $q$ be the associated unique code map from $\mC$ to $\{1\}$. Now, for the morphism $\phi$ to be a neural ring homomorphism, there are three cases. It can map the basis element $\rho_1$ of $\mR_{\{1\}}$ to 0; some basis element $\rho_c$ of $\mR_\mC$; or 1. In the first case where $\phi(\rho_1)=0$, then for any $c\in\mC$,
	$$\rho_1\circ q(c) = \rho_1(1) = 1$$
	$$\text{ but, } \phi(\rho_1) (c)=0(c)=0.$$
	For the second case, let $\phi(\rho_1)=\rho_x$ for some $x\in\mC$. Then, for any $c\in\mC$,
	$$\rho_1^{}\circ q(c) = \rho_1(c) = 1\text{ but, }$$
	$$\big(\phi(\rho_1^{})\big) (c)=\rho_c(v)=
	\begin{cases}
		0 &\quad\text{ if }c\neq v\\ 
		1 &\quad\text{ if }c = v\\
	\end{cases}.$$
	Finally, if $\phi(\rho_1^{})=1$, then
	$$\rho_1\circ q(c) = \rho_1(1) = 1\text{ and, }$$
	$$ \phi(\rho_1)(c)=1(c)=1.$$
	Thus, $\phi$ must map $\rho_1^{}$ to 1, which gives the unique neural ring homomorphism from $\mR_{\{1\}^{}}$ to any other neural ring $\mR_\mC$.
	Next, consider any neural ring homomorphism $\phi: \mR_\mC\to\mR_{\emptyset}$. Let $q$ be the associated unique code map from $\emptyset$ to $\mC$. In fact, $q$ will be the empty function.  So, for the morphism $\phi$ to be a neural ring homomorphism, there is only one option, which is to send all the basis elements $\rho_c^{}$ of $\mR_\mC$ to 1. This gives the unique neural ring homomorphism from any neural ring $\mR_\mC$ to $\mR_{\emptyset}$.
\end{proof}
Consider neural codes $\mathcal{C}_1 \subseteq \{0, 1\}^n$ and $\mathcal{C}_2 \subseteq \{0, 1\}^m$, on $n$ and $m$ neurons respectively to be the objects in $\cod$. Define concatenation product $\mC_1\times\mC_2$ such that the neuron projection($\pi_i$) of codes in $\mC_1\times\mC_2$ onto the $i^{th}$ component gives $\mC_i$. That is $\pi_i(\mC_1\times\mC_2)=\mC_i$ and $\mC_1\times\mC_2$ is a neural code on $n+m$ neurons. For example, if $\mC_1=\{00,10\}$ and $\mC_2=\{100\}$, then their concatenation product $\mC_1\times\mC_2$ is given by $\{00100,10100\}$ where as $\mC_2\times\mC_1 =\{10000,10010\}$. The unique code maps $\pi_i$ from the concatenation product $\mC_1^{}\times \mC_2^{}$ to each $\mC_i^{}$ can be given explicitly by the key vectors (described in the previous section). If $\mC_1^{}$ is on $n$ neurons and $\mC_2^{}$ is on $m$ neurons, then the key vector for $\pi_1$ is $V^1=(1,2,\cdots,n)$ and for $\pi_2$, it will be  $V^2=(n+1,n+2,\cdots,n+m)$.

Next, we define product of code maps, in much of a similar way. Suppose $q_1:\mC_1\to\mD_1$ and $q_2:\mC_2\to\mD_2$ are two code maps with $\mathcal{C}_i \subseteq \{0, 1\}^{n_i}$ and $\mathcal{D}_i \subseteq \{0, 1\}^{m_i}$ for $i=1,2$. Suppose that each $q_i$ is described by a key vector $V^i$. We define $q=q_1\times q_2: \mC_1\times\mC_2\to\mD_1\times\mD_2$ such that $\pi_i q=q_i$. Then $q_1\times q_2$ is also an elementary code map and the vector $V$ that defines it is given by 
$$V_j=
\begin{cases}
	V^1_j, & \text{if } 1\leq j \leq m_1^{}\\
	V^2_j, & \text{if } m_1^{}\leq j \leq m_1^{}+m_2^{}.
\end{cases}
$$

We use this construction to show that the neural category $\mN$ has binary products.

\begin{prop}
	The neural category $\mN$ has binary coproducts.
\end{prop}
\begin{proof}
	Consider any two neural rings $\mR_{\mC_1^{}}$ and $\mR_{\mC_2^{}}$ in $\mN$ with $\mathcal{C}_1 \subseteq \{0, 1\}^n$ and $\mathcal{C}_2 \subseteq \{0, 1\}^m$ on $n$ and $m$ neurons respectively. We define  $\mR_{\mC_1^{}}\oplus\mR_{\mC_2^{}}$ as the neural ring corresponding to the neural code $\mC_1^{}\times \mC_2^{}$. We show it is the binary coproduct of $\mR_{\mC_1^{}}$ and $\mR_{\mC_2^{}}$ in $\mN$. Suppose $\mR_{\mC^{'}}$ be some neural ring and let there be neural ring homomorphisms $\phi_1$ and $\phi_2$ to $\mR_{\mC^{'}}$ from $\mR_{\mC_1^{}}$ and $\mR_{\mC_2^{}}$ respectively. We need to show there is a unique neural ring homomorphism from $\mR_{\mC_1^{}\times\mC_2^{}}$ to $\mR_{\mC^{'}}$. Let $q_1$ and $q_2$ be the corresponding code maps associated to $\phi_1$ and $\phi_2$ respectively. Then we have the following commutative diagram of code maps:
	
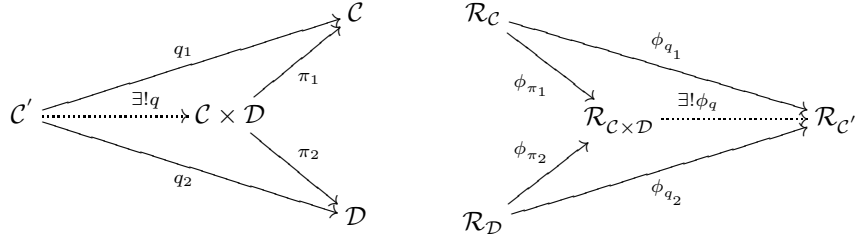
\[
	\xymatrix{
		&&& {\mC}\\
		{\mC^{'}} \ar[drrr]_{q_2} \ar[urrr]^{q_1} \ar@{..>}[rr]^{\hspace{5mm}\exists !q}
		&& {\mC\times\mD} \ar[ur]_{\pi_1} \ar[dr]^{\pi_2}\\
		&&& {\mD}
	}\quad\quad\quad
	\xymatrix{
		{\mR_\mC} \ar[dr]_{\phi^{}_{\pi_1}}\ar[drrr]^{\phi^{}_{q_1^{}}} &&& {}\\
		&{\mR^{}_{{\mC\times\mD}}}  \ar@{..>}[rr]^{\exists !\phi_q\hspace{8mm}}
		&& {\mR_{\mC^{'}}} \\
		{\mR_\mD}\ar[ur]^{\phi^{}_{\pi_2}} \ar[urrr]_{\phi^{}_{q_2^{}}}&&& {}
	}
	\]
	where the key vector of the map $q$ is defined uniquely by $V=(V_1^{1},\cdots,V_n^{1},V_1^{2},\cdots,V_m^{2})$ where $V^{1}$ and $V^{2}$ are the key vectors for $q_1$ and $q_2$ respectively. Corresponding to the code map $q$, we have the associated neural ring homomorphism $\phi_q:\mR_{\mC\times\mD}\to\mR_{\mC^{'}}$, which makes the diagram commute, as corresponding diagram for the associated code maps.
\end{proof}
Moreover, $\cod$ have all finite products (concatenation product $P=\Pi_i C_i$), such that $\pi_i(P)=C_i$.

The category $\cod$ of neural codes does not have coproduct. One way to reason this out is there is no natural way of defining a disjoint union as the length of the strings in the disjoint union may vary. And correspondingly, there is no natural choice of product in $\mN$.

\begin{prop}
	The neural category $\mN$ has all coequalizers. \label{propcoeq}
\end{prop}
\begin{proof}
	Consider the following diagram of neural rings $\mR^{}_{\mC}$ and $\mR^{}_{\mD}$ and neural homomorphisms $\phi$ and $\psi$.
	\[\xymatrix{
		\mR^{}_{\mD}\ar@<0.5ex>[r]^\phi\ar@<-0.5ex>[r]_\psi & \mR^{}_{\mC}
	}\]
	
	Let the corresponding code maps from $\mC$ to $\mD$ be $q_\phi^{}$ and $q_\psi^{}$.  Then we see that $ \E=\{c\in\C\mid q_\phi(c)=q_\psi(c)\} $ together with the inclusion map $ i: \E \mapsto \C $ is the equalizer of $ q_\phi $ and $ q_\psi. $ We claim that  $ \ring{\E} $ together with $ \phi_i: \ring{\C}\to\ring{\E} $ is the coequalizer of $ \phi $ and $ \psi.  $
	
	\no Consider, $$ \phi_i\circ\phi = \phi_{q_{\phi}\circ i }=\phi_{q_{\psi}\circ i}=\phi_i\circ \psi$$	Therefore, $ (\ring{\E},\phi_i) $ coequalizes $ \phi $ and $ \psi. $ Next, we see that the universal mapping property(UMP) follows from UMP of equalizer $ (\E,i). $ Hence the proof. 
\end{proof}
\no An important result in category theory is that finite (co)products and equalizers exists if and only if the category has all (co)limits \cite{leinster2014basic,awodey2010category}. Therefore following theorem is a obvious.  
\begin{theorem}
	The neural category $ \NC $ has all co-limits. Moreover, the category  $ \cod $ has all limits. 
\end{theorem}
\no Next we understand opposite category of any given category $\mathcal{A}$. Refer \cite{awodey2010category} for further details.
\begin{definition}
	The opposite (or “dual”) category $\mathcal{A}^{\text{op}}$ of a category $\mathcal{A}$  has the same objects
	as $\mathcal{A}$, and a morphism $f:C \to D$ in $\mathcal{A}^{\text{op}}$ is an arrow $f:D\to\C$ in $\mathcal{A}$. That
	is, $\mathcal{A}^{\text{op}}$ is just $\mathcal{A}$ with all of the morphisms formally turned around.
\end{definition}
Theorem \ref{thmnrh} connects the categories $\cod$ and $\mN$. However, in the next theorem we will show their exact relationship. Before that we define dual equivalence.
\begin{definition}
	A \textit{dual equivalence}, or \textit{anti equivalence}, between categories $\mathcal{A}$ and $\mathcal{B}$ is simply an equivalence between one and the opposite category of the other, i.e., either $\mathcal{A}$ is equivalent to $\mathcal{B}^{\text{op}}$, or $\mathcal{B}$ is equivalent to $\mathcal{A}^{\text{op}}$.   
\end{definition}
\begin{theorem}
	The categories $\cod$ and $\mN$ are in dual equivalence, i.e., the categories $\mN^{\text{op}}$ and $\cod$ are in equivalence.  \label{equvivalncecat}
\end{theorem}
\begin{proof}
	Construct $ \G':\NC\to\cod $ by associating every neural ring $ \ring{\C} $ to its code $ \C $ and every neural ring homomorphism $ \phi:\ring{\D} \to \ring{\C}$ to its associated code map $ q_\phi:\C \to\D $. Now, construct $\G:\mN^{\text{op}}\to\cod$ where objects are follow same rule as of $\G'$. However, since $\phi:\ring{\C} \to \ring{\D}$ in $\mN^{\text{op}}$ is $\phi:\ring{\D} \to \ring{\C}$ in $\mN$, we have $\G(\phi)=\G'(\phi)=q_\phi:\C \to\D.$ 
	
	Further, by Theorem \ref{thmnrh}, $ \F\circ \G=1_\NC $ and $ \G\circ\F=1_{\cod}. $ Therefore $\mN^{\text{op}}$ and $\cod $ are in equivalence. Hence by definition of dual equivalence we have that $\mN$ and $\cod$ are in dual equivalence. 
\end{proof}
\begin{prop}\cite[Proposition 4.1.11]{leinster2014basic}
	Any set-valued functor with a left adjoint is representable. \label{tomrep}
\end{prop}
\begin{prop}
	The functor $ \G:\NC^{\operatorname{op}} \to \cod $ is a representable functor. 
\end{prop}
\begin{proof}
	Firstly, note that  the equivalence of categories described in Theorem \ref{equvivalncecat} gives a left adjoint functor to $ \G. $ And as $ \cod  $ is a sub-category of SETS the functor $ \G $ satisfies the hypothesis of Proposition \ref{tomrep}. Hence we get that it is a representable functor. 
\end{proof}
\begin{corollary}
	The functor $ \G : \NC^{\operatorname{op}} \to \cod$ preserve limits. 
 \end{corollary}
\begin{proof}
The proof of this corollary follows from the fact that representable functors preserve all the limits.
\end{proof}
\section{Further discussion}
In this section we discuss the categories defined by A Jeffs \cite{jeffs2020morphisms} and the categories in our previous section.  A Jeffs defined and worked with special subsets of a code called trunks. Using these trunks Jeffs defined special morphisms. We will refer to these morphisms as trunk morphisms in this paper. Before we understand trunk morphisms let us look at the definition for trunks given by Jeffs. 
\begin{definition} \cite[Definition 1.1]{jeffs2020morphisms}
	Let $\C$ be a code on $n$ neurons, and let $\sigma\subseteq[n]$. The trunk of $\sigma$ in $\C$ is defined to be the set $$\operatorname{Tk}_\C(\sigma)= \{c\in\C \mid \sigma \subseteq c\}.$$   
\end{definition}
\no A subset of $\C$  is called a trunk in $\C$  if it is empty, or equal to $\operatorname{Tk}_\C(\sigma)$ for some $\sigma  \subseteq  [n]$.
\begin{definition} \cite[Definition 1.2]{jeffs2020morphisms} \label{jeffmorphis}
	Let $\C$  and $\D$  be codes. A function $f : \C  \rightarrow  \D $ is a trunk morphism if for every
	trunk $T \subseteq  \D$  the pre-image $f^{-1} (T )$ is a trunk in $\C$. A trunk morphism will be a trunk isomorphism if it has an inverse which is also a trunk morphism. 
\end{definition}
 We observed that the elementary code maps of Theorem \ref{thmnrh} also satisfies the above definition. So, given any neural ring homomorphism $\phi:\ring{\D} \to \ring{\C}$, the associated code map $q_\phi: \C\to \D$ is a trunk morphism.

 Jeffs, later defined a category, denoted by \textbf{Code} with neural codes as objects and with morphisms as trunk morphisms.  We observe that $\cod$ is subcategory of \textbf{Code}. Clearly it is not full as there are trunk morphisms which are not morphisms in $\cod$. 
 \begin{eg}
 	Let $\C=\{000,100,001,110,011\}$ and $\D=\{0000,1000,0010,1100,0011\}$. Consider the map from $q:\C\to \D$ given by $000\mapsto 0000, 100 \mapsto 1000, 001 \mapsto 0010, 110\mapsto 1100$ and $011\mapsto 0011$. We observe that this is a trunk morphisms. However $q$ is not composition of elementary code maps since we cannot write a key vector to it.    
 \end{eg}
 
 Next, Jeffs defined a category \textbf{NRing} with neural rings as objects. The morphisms in this category are monomial maps which he defines as follows: 
 \begin{definition}\cite[Definition 6.2]{jeffs2020morphisms}
 	Let $\ring{\C}$ and $\ring{\D}$ be neural rings with coordinates $\{x_1,\dots,x_n\}$ and $\{y_1,\dots,y_m\},$ respectively. A monomial map from $\ring{\C}$ to $\ring{\D}$ is a ring homomorphism $\phi:\ring{\C}$ to $\ring{\D}$ with the property that if $p\in\ring{\C}$ is a monomial in the $x_i$, then $\phi(p)$ is a monomial in the $ y_j$ or it is zero. 
 \end{definition}
 Jeffs showed that \textbf{NRing} becomes a category. We observe that given a neural ring homomorphism it is a monomial maps. However, there are monomial maps which are not neural ring homomorphisms. So, we have that $\mN$ is a subcategory of \textbf{NRing} and it is clearly not full. Further, Jeffs showed that these the categories \textbf{NRing} and \textbf{Code} are in dual equivalence. The products of codes defined by Jeffs \cite{jeffs2020morphisms} become the products of the category \textbf{Code}.  We observe the following for the category \textbf{NRing}:
 
  \begin{enumerate}
  	\item The objects $\ring{\emptyset}$ and $\ring{\{1\}}$ in \textbf{NRing} are final and initial objects in the category. The proof of this is similar to Proposition \ref{propintial}. \item Furthermore, the category \textbf{NRing} has all coequalizers and can be shown similarly as in  Proposition \ref{propcoeq}. 
  \end{enumerate}

		\bibliographystyle{plain}
	\bibliography{refs}
Authors Affiliation:  \vspace{0.2cm}\\
Neha Gupta \\ Assistant Professor \\ Department of Mathematics \\ Shiv Nadar Institution of Eminence (Deemed to be University) \\ Delhi-NCR \\ India \\ Email: neha.gupta@snu.edu.in  \vspace{0.6cm}\\
Suhith K N \\ Research Scholar \\ Department of Mathematics \\ Shiv Nadar Institution of Eminence (Deemed to be University) \\ Delhi-NCR \\ India \\ Email: sk806@snu.edu.in \\
Suhith's research is partially supported by Inspire fellowship from DST grant IF190980.
\end{document}